\documentclass[11pt]{article}

\usepackage[left=0.90in,top=1.0in,right=0.90in,bottom=0.85in]{geometry}

\usepackage[dvipsnames]{xcolor}
\newcommand{\JG}[1]{{\color{Green}JG: #1}}

\usepackage{graphics}
\usepackage{amssymb}
\usepackage{amsmath}
\usepackage{soul}
\usepackage{cite}
\usepackage{color}
\usepackage{xspace}
\usepackage{caption}
\usepackage{bbm}
\usepackage{comment}
\usepackage{algorithmicx}
\usepackage{subfigure,float}
\usepackage{psfrag}
\usepackage{wrapfig}
\usepackage{soul}
\usepackage{tikz}
\usetikzlibrary{shapes,arrows}
\usetikzlibrary{positioning, decorations.markings}
\usetikzlibrary{calc}
\usepackage{tcolorbox}

\usetikzlibrary{decorations.pathreplacing}
\allowdisplaybreaks  
\usetikzlibrary{arrows}
\tikzstyle{block}=[draw opacity=0.7,line width=1.4cm]

\usepackage{amsthm}

\newenvironment{proof}[1][Proof]%
  {\smallskip\par\noindent\textbf{#1\,:\ }}%
  {\hspace*{\fill} \rule{6pt}{6pt}\smallskip}
\newenvironment{proof*}[1][Proof]%
  {\medskip\par\noindent\textbf{#1\,:\ }}%

\title{\LARGE \bf High-Resolution Modeling of the Fastest First-Order Optimization Method for Strongly Convex Functions}

\author{Boya Sun, Jemin George and Solmaz Kia
}

\newcommand{\real}{{\mathbb{R}}} \newcommand{\reals}{{\mathbb{R}}}
 
\newcommand{\realpositive}{{\mathbb{R}}_{>0}}

\newcommand{\argmin}{\operatorname{argmin}}

 \newcommand{\boxend}{\hfill \ensuremath{\Box}}

\newcommand{\solmaz}[1]{{\color{red}#1}}
\newcommand{\Boya}[1]{{\color{brown}#1}}

\newtheorem{thm}{Theorem}[section]

\newtheorem{rem}{Remark}[section]

\newtheorem{lem}{Lemma}[section]

\newcommand{\oprocendsymbol}{\hbox{$\bullet$}}
\newcommand{\oprocend}{\relax\ifmmode\else\unskip\hfill\fi\oprocendsymbol}

\begin{document}
\maketitle
\thispagestyle{empty}
\pagestyle{empty}

\begin{abstract}
Motivated by the fact that the gradient-based optimization algorithms can be studied from the perspective of limiting ordinary differential equations (ODEs), here we derive an ODE representation of the accelerated triple momentum (TM) algorithm. For unconstrained optimization problems with strongly convex cost, the TM algorithm has a proven faster convergence rate than the Nesterov's accelerated gradient (NAG) method but with the same computational complexity. We show that similar to the NAG method to capture accurately the characteristics of the TM method, we need to use a high-resolution modeling to obtain the ODE representation of the TM algorithm. We use a Lyapunov analysis to investigate the stability and convergence behavior of the proposed high-resolution ODE representation of the TM algorithm. We show through this analysis that this ODE model has robustness to deviation from the parameters of the TM algorithm. We compare the rate of the ODE representation of the TM method with that of the NAG method to confirm its faster convergence. Our study also leads to a tighter bound on the worst rate of convergence for the ODE model of the NAG method. Lastly, we discuss the use of the integral quadratic constraint (IQC) method to establish an estimate on the rate of convergence of the TM algorithm. A numerical example demonstrates our results.
\end{abstract}

\section{Introduction}\label{sec::Intro}

During the past decade, we have witnessed a surge in the design of first-order gradient descent algorithms with parallel/decentralized/distributed structure that is intended to address the optimization problems that arise in large-scale machine learning with stringent computation/communication/storage requirements~\cite{AN-AO:09,BJ-MR-MJ:09,SB-NP-EC-BP-JE:10,MZ-SM:09c,JD-AA-MW:12,SSK-JC-SM:15-auto,SSK:17}. However, in many of the applications involving large-scale optimizations such as operational decision-making for networked systems, there is a need for real-time adjustment of the system’s response/decision to the present situation. Therefore, besides the need for efficiency in resource (computation/communication/storage) management, fast converging optimization algorithms for large-scale problems are now more and more in demand.

As it has been known in the classical optimization literature, improvement to the rate of convergence of optimization algorithms within a first-order framework can be obtained through methods such as quasi-Newton~\cite{DB2019,DGL2016},  Polyak's heavy-ball~\cite{POLYAK19641,EG-HRF:15}, and Nesterov's accelerated gradient (NAG)~\cite{YN:08,nesterov2013introductory} methods. Among these methods, because of its simple structure and its global convergence guarantees for convex objective functions, NAG has received much attention in the optimization and machine learning community. However, the quest for alternative fast converging first-order optimization algorithms is still an ongoing research topic. Recently, a new accelerated gradient-based method called the Triple Momentum (TM) method, which has the same computational complexity as the NAG method but with a proven faster convergence rate, was proposed in~\cite{Scoy18}. Our objective in this paper is to obtain a high-resolution continuous-time representation for the TM method and study its stability and convergence via control theoretic tools.

ODE representation and its analysis for optimization algorithms in the continuous-time domain have a long history going back to~\cite{KJA-LH-HU:58}; more discussions can be found in~\cite{UH-JM:96,Schropp2000, Jongen2004, Shikhman2009, helmke2012optimization}. Continuous-time modeling comes with ease in theoretical analysis via powerful control theoretic tools such as Lyapunov analysis, perturbation theory, and the integral quadratic constraint (IQC) methods. Also, the continuous-time perspective provides intuition to design new algorithms, especially arriving at distributed algorithms in a systematic way from centralized solutions. Furthermore, the convergence analysis of several gradient-based Markov Chain Monte Carlo sampling schemes relies on the continuous-time approximation of such algorithms \cite{cheng2018convergence, ma2019sampling}.  Therefore, recently, ODE modeling has regained popularity to address the need to design new distributed gradient descent based optimization algorithms~\cite{JW-NE:11,JL-CYT:12,SSK-JC-SM:15-auto,DV-FZ-AC-GP-LS:15}, as well as to analyze the new accelerated optimization algorithms~\cite{Boyd:16,2016arXiv161102635W,AW-ACW-MIJ:16,BS:18,Guilherme2019arXiv}. In \cite{Boyd:16}, a second-order ODE is presented as the limit of the NAG method. The connection between ODEs and discrete optimization algorithms is further strengthened in \cite{2016arXiv161102635W} by establishing an equivalence between the estimate sequence technique and Lyapunov function techniques. In \cite{AW-ACW-MIJ:16}, the authors propose a variational, continuous-time framework for understanding accelerated methods and show that there is a Lagrangian functional that generates a large class of accelerated methods in continuous time. NAG method and many of its generalizations can be viewed as a systematic way to go from the continuous-time curves generated by the Lagrangian functional to a family of discrete-time accelerated algorithms~\cite{AW-ACW-MIJ:16}. An ODE-based analysis of mirror descent given in \cite{KricheneNIPS2016} delivers new insights into the connections between acceleration and constrained optimization, averaging, and stochastic mirror descent. A deeper insight into the acceleration phenomenon via high-resolution ODE representation of various first-order methods is presented in \cite{BS:18}. These high-resolution ODEs are shown to permit a general Lyapunov function framework for convergence analysis in both continuous and discrete time~\cite{BS:18}. Finally, in \cite{Guilherme2019arXiv}, the authors show that different types of proximal optimization algorithms based on fixed-point iteration can be derived from the gradient flow by using splitting methods~for~ODEs.

The connection between ODE representation of optimization algorithms and their discrete-time counterpart is often established by taking the step size of the discrete-time algorithm to zero and deriving a limiting ODE using first-order derivatives modeling. This approach works well for gradient descent and Newton algorithms (thus obtaining $\dot{x}=-\nabla f(x)$ and $\dot{x}=-\nabla^2 f(x)^{-1}\nabla f(x)$ from $x(k+1)=x(k)-s\nabla f(x)$ and $x(k+1)=x(k)-s\nabla^2 f(x(k))^{-1}\nabla f(x(k))$, respectively, where $s$ is the step size). However, recent literature has shown that first-order ODE modeling of accelerated algorithms such as the Polyak's heavy-ball and NAG methods fails to capture the true behavior of these algorithms~\cite{BS:18}. In fact, it has been shown that the first-order ODE representation cannot differentiate between these two algorithms since it yields an identical limiting equation for both. Recent literature, therefore, has looked at second-order ODE representation of these algorithms~\cite{Boyd:16,BS:18}. These high-resolution ODEs are more accurate since they better capture the characterizations of the discrete-time accelerated methods in their continuous-time counterpart representations. 

In this paper, we derive a second-order ODE representation for the accelerated TM method and show that the high-resolution ODE is able to accurately capture the characterizations of the  TM method. For clarity, hereafter we refer to the TM method of~\cite{Scoy18} as the discrete-time TM. We present a Lyapunov analysis to study the stability and convergence behavior of the resulted ODE TM representation. We use our Lyapunov analysis to show that the TM method has robustness with respect to deviation from its parameters. We also use our framework to estimate the rate of convergence of the TM algorithm and compare it to the NAG method, which confirms its faster convergence. Our work also leads to a tighter estimate on the rate of convergence of the ODE representation of the NAG method. We also present an IQC framework to establish a bound on the rate of convergence of the algorithm. Using a numerical example, we show the accuracy of our second-order ODE representation in capturing the accelerated behavior of the TM method and its faster convergence over the high-resolution ODE representation of the NAG method given in~\cite{BS:18}.

\emph{Notations}: $\reals$ and $\realpositive$ are the set of real and positive real numbers. ${A}^\top$ is the transpose of matrix ${A}$.
We let  ${0}_{n}$  denote the vector of $n$ zeros and 
$I_n$ denote the $n\times n$ identity~matrix.  When clear from the context, we do not
specify the matrix dimensions. For a vector $x\in\reals^n$, $\|{x}\|=\sqrt{{x}^\top{x}}$ is the standard Euclidean norm. The gradient of $f: \reals^d\to\reals$, is denoted by $\nabla f({x})$.  The following relations hold for a differentiable function $f: \reals^d\to\reals$ that is $M$-strongly convex, $M\in\realpositive$, over $\real^d$ ,  
\begin{subequations}\label{eq::MCon}
\begin{align}
   &\!\!\!f(\mathsf{y})\!-\!f(\mathsf{x})\!\leq\!\nabla f(\mathsf{x})\!^\top\!(\mathsf{y}\!-\!\mathsf{x})\!+\!\frac{1}{2M}\|\nabla f(\mathsf{y})\!-\!\nabla f(\mathsf{x})\|^2\!,\label{eq::MCon-1}\\
   &\!\!\!M\|\mathsf{y}-\mathsf{x}\|^2\leq  ({\mathsf{y}}-{\mathsf{x}})^\top(\nabla f({\mathsf{y}})-\nabla f({\mathsf{x}})),\label{eq::MCon-2}\\
  & \!\!\!M\|\mathsf{y}-\mathsf{x}\|\leq  \|\nabla f({\mathsf{y}})-\nabla f({\mathsf{x}})\|,\label{eq::MCon-3}
\end{align}
\end{subequations}
for any ${\mathsf{x}},{\mathsf{y}}\in \real^d$~\cite{XZ:18}. 
When $\nabla f:\mathbb{R}^d\to \mathbb{R}^d$ of a convex function $f: \real^d\to\real$ is
$L$-Lipschitz continuous,~$L\in\realpositive$, i.e., $\|\nabla f(\mathsf{y})-\nabla f(\mathsf{x})\|\leq L\|\mathsf{y}-\mathsf{x}\|$, we have
\begin{subequations}\label{eq::Lip}
\begin{align}
   &\!\!\! f(\mathsf{y})-f(\mathsf{x})\leq\nabla f(\mathsf{x})^\top\!(\mathsf{y}-\mathsf{x})\!+\!\frac{L}{2}\|\mathsf{y}\!-\!\mathsf{x}\|^2,\label{eq::Lip-2}\\
 & \!\!\!  f(\mathsf{y})\!-\!f(\mathsf{x})\geq\nabla f(\mathsf{x})^\top\!(\mathsf{y}-\mathsf{x})\!+\!\frac{1}{2L}\|\nabla f(\mathsf{y})\!-\!\nabla f(\mathsf{x})\|^2\!,\label{eq::Lip-3}
\end{align}
\end{subequations}
for all $\mathsf{x},\mathsf{y} \in \mathbb{R}^d\times \mathbb{R}^d$~\cite{XZ:18}.
We represent the class of differentiable and $M$ strongly convex functions whose gradient is L-Lipschitz with $\mathcal{S}_{M,L}$.

\section{Problem definition}\label{sec::Prob_Def}
Consider 
\begin{equation}\label{Eqn:Opt}
    \mathsf{x}^\star=\argmin\limits_{x\in \mathbb{R}^n}\, f(x),
\end{equation}
where $f:\mathbb{R}^n\rightarrow \mathbb{R}$ and $f\in\mathcal{S}_{L,M}$. We assume that $\mathsf{x}^\star$ exists and is reachable. The minimizer of this optimization problem is specified as follows.

\begin{lem}[Minimizer of~\eqref{Eqn:Opt}~\cite{DB2019}] Consider optimization problem~\eqref{Eqn:Opt}. A point $\mathsf{x}^\star \in \mathbb{R}^n$ is a unique solution of~\eqref{Eqn:Opt} if and only if
$\nabla f(\mathsf{x}^\star)=0$.
\end{lem}
In what follows, we let 
\begin{align}\label{eq::kappa_rho}
    \kappa=\frac{L}{M},\quad \rho=1-\frac{1}{\sqrt{\kappa}}.
\end{align}
We refer to $\kappa$ as the condition number of the cost function~$f$.

\subsection{Discrete-time TM Method}

Here we consider the TM method, proposed in~\cite{Scoy18} as the fastest known globally convergent first-order method for solving strongly convex optimization problems. The TM method is an accelerated gradient-based optimization algorithm given as 
\begin{subequations}\label{methods}
\begin{align}
\epsilon_{k+1}&=(1+\beta)\epsilon_k-\beta\epsilon_{k-1}-\alpha\nabla f(y_k),\label{recur1}\\
y_k&=(1+\gamma)\epsilon_k-\gamma\epsilon_{k-1},\label{recur2}\\
x_k&=(1+\delta)\epsilon_k-\delta\epsilon_{k-1},\label{recur3}
\end{align}
\end{subequations}
where the algorithm parameters are given as (recall~\eqref{eq::kappa_rho})
\begin{align}\label{eq::TM_param}
(\alpha, \beta, \gamma, \delta ) = \left(\frac{1+\rho}{L}, \frac{\rho^2}{2-\rho}, \frac{\rho^2}{(1\!+\!\rho)(2\!-\!\rho)}, \frac{\rho^2}{1\!-\!\rho^2} \right),
\end{align}
and 
$\epsilon_0, \epsilon_{-1} \in \mathbb{R}^n$ are the initial conditions, $x \in \mathbb{R}^n$ is the output. 
The TM method has the same numerical complexity as the NAG method but converges faster. In \cite{Scoy18}, it is shown that starting from any initial conditions $\epsilon_0, \epsilon_{-1} \in \mathbb{R}^n$, the trajectories of $x_k, y_k,\epsilon_k$ converge to $\mathsf{x}^\star$ with the same rate but the convergence error of output $x$ is smaller. We observe the same trend in the high-resolution ODE representation of the TM method; see Section~\ref{sec:num} for numerical examples.

Our objective in this paper is to derive a high-resolution ODE representation of the TM algorithm that accurately captures the performance characteristics of its discrete-time counterpart and establish its formal convergence guarantees using the Lyapunov stability analysis. To facilitate our discussions given next, we define a function $\mu(\alpha,\beta) : \realpositive \times \mathbb{R} \rightarrow \realpositive$ (or simply $\mu$) as
\begin{equation}\label{Eqn:mu}
\mu(\alpha,\beta)=\left(\frac{1-\beta}{\sqrt{\alpha}(1+\beta)}\right)^2.    
\end{equation} 
Using the parameter relations given in \eqref{eq::kappa_rho} and \eqref{eq::TM_param} for the TM method, $\mu=\mu(\alpha,\beta)$ can also be written as
 \begin{align}\label{eq::mu_kappa_L}
&\mu = \frac{(9\,\kappa^{2}\sqrt{\kappa}-6\,\kappa^2+\kappa\sqrt{\kappa})L}{8\,\kappa^{3}\sqrt{\kappa}-12\,\kappa^3+14\,\kappa^2\sqrt{\kappa}-9\,\kappa^2+4\,\kappa\sqrt{\kappa}-\kappa}.
 \end{align}
 As shown in Fig.~\ref{Fig::mu}, the maximum value of $\mu$ is $L$, which~is attained at  $\kappa\!=\!1$. When $\kappa\!\to\!\infty$, $\mu\!\to\!0$. We can also show~that 
 \begin{align}\label{eq::mu-L}
      \mu(\alpha,\beta) \in (0,L].
 \end{align}
Replacing $L$ with $\kappa M$ in~\eqref{Eqn:mu}, we can also show that
  \begin{align}\label{eq::mu_bound_M}
     \mu(\alpha,\beta) \in [M,1.3661M].
 \end{align}

\begin{rem}[Role of parameter $\mu$]\rm{
Parameter $\mu$ plays a vital role in the analysis of the high-resolution ODE representation of the TM method. Also note that after substituting appropriate $\alpha$ and $\beta$ values into \eqref{Eqn:mu}, we obtain $\mu_{_{\text{NAG}}} \textnormal = M$ for NAG method while $\mu \geq M$ for the TM method. Therefore the parameter $\mu$ also plays an important role when comparing the convergence rate between the high-resolution NAG and TM methods. }\boxend
\end{rem}
\begin{figure}[t]
  \centering  
  \includegraphics[trim=5pt 30pt 0pt 30pt, clip,scale=0.73]{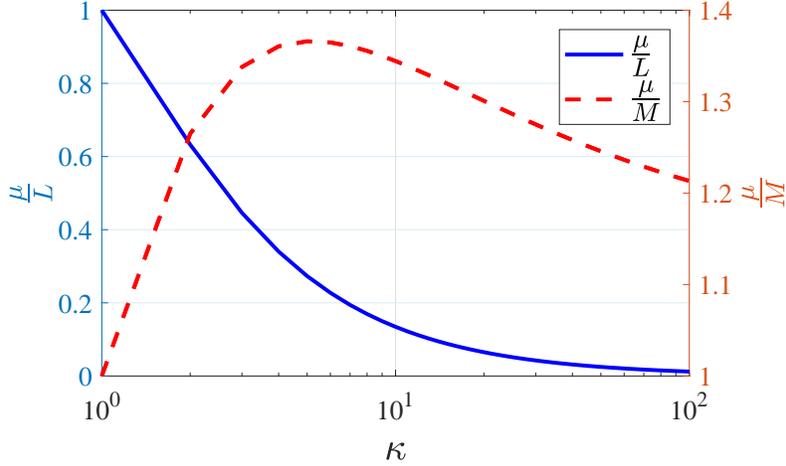}
  \caption{Variation of $\mu/L$ (solid blue) and $\mu/M$ (dashed red) with $\kappa$.}
  \label{Fig::mu} 
\end{figure}
\section{Continuous-Time Representation of the TM Method}
Let $t_k=k\sqrt{\alpha}$ and $y_k=Y(t_k)$ for some sufficiently smooth curve $Y(t)$. Now the Taylor series expansion at both $y_{k+1}$ and $y_{k-1}$ with the step size $\sqrt{\alpha}$ are
\begin{align}\label{Eqn:3rdTaylor1}
\begin{split}
 &y_{k+1}=Y(t_{k+1})=Y(t_k)+\Dot{Y}(t_k)\sqrt{\alpha}+\frac{1}{2}\Ddot{Y}(t_k)(\sqrt{\alpha})^2+\mathcal{O}((\sqrt{\alpha})^3), 
\end{split}
\end{align}
\begin{align}\label{Eqn:3rdTaylor2}
\begin{split}
&y_{k-1}=Y(t_{k-1})=Y(t_k)-\Dot{Y}(t_k)\sqrt{\alpha}+\frac{1}{2}\Ddot{Y}(t_k)(\sqrt{\alpha})^2+\mathcal{O}((\sqrt{\alpha})^3).
\end{split}
\end{align}
Combining \eqref{Eqn:3rdTaylor1} and \eqref{Eqn:3rdTaylor2} yields
    \begin{align}\label{Eqn:3rdTaylor3}
        Y(t_{k+1})+Y(t_{k-1})-2Y(t_k) = \alpha\ddot{Y}(t_k) + \mathcal{O}(\alpha^2). 
    \end{align}

Next, we note that we can rewrite~\eqref{methods} as 
\begin{subequations}
\begin{align}
\epsilon_{k+1}&=\epsilon_k +\beta \left(\epsilon_k-\epsilon_{k-1}\right)-\alpha\nabla f(y_k),\label{Eqn:TM1a}\\
y_k &= \epsilon_k + \gamma\left(\epsilon_k-\epsilon_{k-1}\right),\label{Eqn:TM1b}\\
x_k &= \epsilon_k + \delta\left(\epsilon_k-\epsilon_{k-1}\right).\label{Eqn:TM1c}
\end{align}
\end{subequations}
From \eqref{Eqn:TM1a}, we obtain
\begin{align}\label{Eqn:TM1a11}
\beta \left(\epsilon_{k-1}-\epsilon_k\right) + \left(\epsilon_{k+1}-\epsilon_k\right) +\alpha\nabla f(y_k) = 0.
\end{align}
Now adding and subtracting $\beta\left(\epsilon_{k+1}-\epsilon_k\right)$  and dividing both sides of \eqref{Eqn:TM1a11} with ${\beta}{\alpha}$ yields 
\begin{align*}
\frac{\left(\epsilon_{k+1}\!+\!\epsilon_{k-1}\!-2\epsilon_k\right)}{\alpha} \!+\! \frac{(1-\beta)}{\beta\alpha} \left(\epsilon_{k+1}-\epsilon_k\right) \!+\!\frac{1}{\beta}\nabla f(y_k) = 0.
\end{align*}

Substituting $\epsilon_k=\varepsilon(t_k)$ and $y_k=Y(t_k)$ and \eqref{Eqn:3rdTaylor3} yields
\begin{align}\label{Eqn:HB2}
&\ddot{\varepsilon}(t_k) + \mathcal{O}(\alpha) + \frac{1-\beta}{\beta\sqrt{\alpha}}\left( \dot{\varepsilon}(t_k)+\frac{1}{2}\Ddot{\varepsilon}(t_k)\sqrt{\alpha}+\mathcal{O}({\alpha}) \right)+\frac{1}{\beta} \nabla f(Y(t_k)) = 0,
\end{align}
where we used 
$$\varepsilon(t_{k+1})-\varepsilon(t_k) = \dot{\varepsilon}(t_k)\sqrt{\alpha}+\frac{1}{2}\Ddot{\varepsilon}(t_k)(\sqrt{\alpha})^2+\mathcal{O}((\sqrt{\alpha})^3).$$
If we consider the limit of~\eqref{Eqn:HB2}, when $\alpha \to 0$, we then obtain the low-resolution representation for the TM algorithm as 
\begin{align}\label{Eqn:LowResTM}
     \ddot{\varepsilon}(t_k) + 2\sqrt{
     \mu}\, \dot{\varepsilon}(t_k)  +  \nabla f(Y(t_k)) = 0,
\end{align}
where we used the relation $\beta=\displaystyle\frac{1-\sqrt{\mu\alpha}}{1+\sqrt{\mu\alpha}}$. The low-resolution representation in \eqref{Eqn:LowResTM} is exactly the same as the low-resolution ODE obtained for the NAG and heavy ball methods in~\cite{BS:18}. Therefore the low-resolution ODE fails to distinguish the TM method from the NAG and heavy ball methods. Next, we derive a high-resolution ODE that captures the characteristics of the TM method, i.e., shows a faster convergence in comparison to the NAG and heavy ball methods.

\subsection{High-resolution ODE of TM Method}
We obtain a high-resolution ODE for the TM method by ignoring $\mathcal{O}(\alpha)$ terms but keeping $\sqrt{\alpha}$ in~\eqref{Eqn:HB2}, which results~in 
\begin{align}\label{Eqn:HB3}
    \frac{1+\beta}{2\beta}\ddot{\varepsilon}(t_k) + \frac{1-\beta}{\beta\sqrt{\alpha}} \dot{\varepsilon}(t_k)  + \frac{1}{\beta} \nabla f(Y(t_k)) = 0.
\end{align}
Now multiplying both sides of~\eqref{Eqn:HB3} by $\displaystyle\frac{2\beta}{1+\beta}$ and substituting \eqref{Eqn:mu} yields
\begin{align*}
    \ddot{\varepsilon}(t_k) + 2\sqrt{\mu} \dot{\varepsilon}(t_k)  + \left(1+\sqrt{\mu\alpha}\right) \nabla f(Y(t_k)) = 0,
\end{align*}
where we used $\displaystyle \frac{2}{1+\beta} = 1+\sqrt{\mu\alpha}$. 
Next, we note that from \eqref{Eqn:3rdTaylor2} we have  
\begin{align}\label{Eqn:DiffEpsilon}
    \varepsilon(t_k) - \varepsilon(t_{k-1}) = \dot{\varepsilon}(t_k)\sqrt{\alpha} + \mathcal{O}(\alpha).
\end{align}
Ignoring the $\mathcal{O}(\alpha)$ term and substituting \eqref{Eqn:DiffEpsilon} into \eqref{Eqn:TM1b} yields $ Y=\varepsilon+\sqrt{\alpha}\gamma\dot{\varepsilon}$.
Let $x_k=X(t_k)$. Similarly, from \eqref{Eqn:TM1c} we have $X=\varepsilon+\sqrt{\alpha}\delta\dot{\varepsilon}$. Thus, we obtain 
\begin{subequations}\label{eq::TM_ODE}
\begin{align}
 & \ddot\varepsilon+2\sqrt{\mu}\,\dot{\varepsilon}+(1+\sqrt{\mu\alpha})\nabla f(Y)=0,\label{eq::ODE_1}\\
 &     Y=\varepsilon+\sqrt{\alpha}\gamma\,\dot{\varepsilon},\label{eq::ODE_2}\\
 &   X=\varepsilon+\sqrt{\alpha}\delta\,\dot{\varepsilon}.\label{eq::ODE_3}
\end{align}
\end{subequations}
 as a high-resolution ODE that maintains the main characteristics of the TM method with the appropriate initial conditions $\varepsilon_0$ and $Y_0$. Note that differentiating \eqref{eq::ODE_1} yields
\begin{align}
    \dddot\varepsilon+2\sqrt{\mu}\ddot{\varepsilon}+(1+\sqrt{\mu\alpha})\nabla^2 f(Y)\dot{Y}=0.\label{eq::ODE_1a}
\end{align}
Then, substituting \eqref{eq::ODE_2}, and its first and second derivative  $\dot{Y}=\dot{\varepsilon}+\sqrt{\alpha}\gamma\ddot{\varepsilon}$ and $\ddot{Y}=\ddot{\varepsilon}+\sqrt{\alpha}\gamma\dddot{\varepsilon}$ into \eqref{eq::ODE_1a} yields the high-resolution representation of the TM method in terms of output $Y$ as
\begin{align}\label{eq::ODE_Y}
\begin{split}
    &\ddot{Y}+2\sqrt{\mu}\dot{Y}\!+\!\gamma(1+\sqrt{\mu\alpha})\sqrt{\alpha}\nabla^2f(Y)\Dot{Y}+(1+\sqrt{\mu\alpha})\nabla f(Y)=0.
    \end{split}
\end{align}

In what follows,  we use ~\eqref{eq::ODE_Y} to analyze the stability and convergence of the ODE representation of the TM method in~\eqref{eq::TM_ODE} and compare its rate of convergence to that of the high-resolution ODE representation of the NAG method given in~\cite{BS:18} as
\begin{align}\label{eq::ODE_Y_NAG}
    &\!\!\!\ddot{Y}\!+\!2\sqrt{M}\dot{Y}+\sqrt{s}\nabla^2f(Y)\Dot{Y}\!+\!(1+\sqrt{M s})\nabla f(Y)\!=\!0,
\end{align}
where $s=\frac{1}{L}$. One can think of $s$ as the equivalent of $\alpha$ in the TM method~\eqref{methods}, i.e., it is the step-size multiplying the gradient term. In comparing the TM method to the  NAG method, it is interesting to recall~\eqref{eq::mu_bound_M}. It is important to note that the main difference between the NAG method given in \eqref{eq::ODE_Y_NAG} and the TM methods in \eqref{eq::ODE_Y} is in the coefficient multiplying the gradient correction term $\nabla^2f(Y)\Dot{Y}$. Even though it is not discussed in~\cite{BS:18}, it is worth mentioning that by introducing an appropriate intermediate variable similar to~\eqref{eq::ODE_2}, one can write the  NAG method in an equivalent form that does not require $\nabla^2f(Y)$. 

In the ODE representation of the TM and NAG algorithms we also refer to the parameters $\alpha$ and $s$ as~stepsize.
\section{Convergence Analysis}
In this section, we present a detailed convergence analysis of~\eqref{eq::TM_ODE} and establish the convergence rate of the algorithm. We start by   identifying the equilibrium point of~\eqref{eq::TM_ODE}.

\begin{lem}[Equilibrium Point of~\eqref{eq::TM_ODE}] \label{lem::eq}
Assume $f$ is strongly convex and continuously differentiable. Then,~\eqref{eq::ODE_1} has a unique equilibrium point given by  
\begin{align*}
    s_{\textup{eq}}=\big\{(\dot{\varepsilon},\varepsilon)\in\real^{n}\times \real^n\,|\,\dot{\varepsilon} = 0, \nabla f(\varepsilon)=0\big\}.
\end{align*}
Moreover, $Y$ and $X$ at the equilibrium point both satisfy $\nabla f(Y_{\textup{eq}})=\nabla f(X_{\textup{eq}})=0$. 
\end{lem}
\begin{proof}
To obtain the equilibrium point of~\eqref{eq::ODE_1}, we set $\ddot{\varepsilon}\!=\!\dot{\varepsilon} \!=\! 0$. Then, it follows from~\eqref{eq::ODE_2} and~\eqref{eq::ODE_3} that $\varepsilon_{\textup{eq}} = Y_{\textup{eq}}\!=\! X_{\textup{eq}}$. As a result, at the equilibrium point,  from~\eqref{eq::ODE_1} we obtain $ \nabla f(\varepsilon_{\textup{eq}})\!=\!0$, and thereby  $\nabla f(Y_{\textup{eq}})\!=\!\nabla f(X_{\textup{eq}})\!=\!0$.
\end{proof}

The next result establishes exponential stability of~\eqref{eq::TM_ODE} and gives an estimate on its rate of convergence.  

\begin{thm}[Stability and convergence analysis the ODE TM]\label{thm::main}
Consider the optimization problem~\eqref{Eqn:Opt} and the algorithm~\eqref{eq::TM_ODE}. (a) For  $\alpha,\beta,\gamma,\delta\in\real_{>0}$, $\beta\neq1$, starting from any initial condition $\varepsilon(0),\dot{\varepsilon}(0)\in\real^n$ the trajectories of $t\mapsto\varepsilon$, $t\mapsto X$ and $t\mapsto Y$ converge exponentially fast 
to $\mathsf{x}^\star$, the minimizer of~\eqref{Eqn:Opt}. Moreover, $f(Y)-f(\mathsf{x}^\star)$ vanishes exponentially with a rate no worse than  $ p^\star$
where
\begin{align}\label{eq::p_TM}
&p^\star=\max_{\phi\in\real_{>0}} p(\phi),\\
&p(\phi) = \min {\small\left\{{\frac{\sqrt{\mu}}{2},\frac{3L}{4\kappa(1+\phi)\sqrt{\mu}},\frac{1}{\gamma\sqrt{\alpha}(1+\frac{1}{\phi})},\frac{4\sqrt{\mu}}{3+\frac{2}{\phi}}}\right\}}.\nonumber
\end{align}
(b) If $\alpha,\beta,\gamma,\delta>0$ is set to the parameters of the TM method in~\eqref{eq::TM_param} and the algorithm is initialized at 
\begin{subequations}\label{eq::initial_TM}
\begin{align}
\varepsilon(0) &= Y_0-\frac{\alpha\gamma^2(1+\sqrt{\mu\alpha})\nabla f(Y_0)}{(1-2\gamma\sqrt{\mu\alpha})},\\
    \dot{\varepsilon}(0) &= \frac{\sqrt{\alpha}\gamma(1+\sqrt{\mu\alpha})\nabla f(Y_0)}{(1-2\gamma\sqrt{\mu\alpha})},
\end{align}
\end{subequations}
 where $Y_0 = Y(0) \in\real^n$, then  the trajectory $t\mapsto Y$ of~\eqref{eq::TM_ODE} satisfies 
\begin{align*}
     f(Y(t))-f(\mathsf{x}^\star)\leq \frac{1.5\|Y_0-\mathsf{x}^\star\|}{\alpha}^2\textup{e}^{-p_{_{\textup{TM}}}^\star t},\qquad t\in\real_{\geq0},
\end{align*}
where $p_{_{\textup{TM}}}$ is $p^\star$, evaluated at $\mu$ given by~\eqref{eq::mu_kappa_L}, and   $\alpha$ and $\gamma$ of the TM method. 
\end{thm}

\begin{proof}
We first note that given $\beta\neq 1$, by definition~\eqref{Eqn:mu}, we have $\mu\in\real_{>0}$.
Next, recall~\eqref{eq::ODE_Y}, the equivalent ODE representation of~\eqref{eq::TM_ODE} in terms of output $Y$. As shown in Lemma~\ref{lem::eq}, the equilibrium point $Y_{\textup{eq}}$ of~\eqref{eq::ODE_Y} satisfies $Y_{\textup{eq}}=\mathsf{x}^{\star}$. To study the convergence of~\eqref{eq::ODE_Y} to $\mathsf{x}^\star$, we consider the radially unbounded Lyapunov function candidate 
\begin{align}\label{eq::v}
    V(t)=&(1+\sqrt{\mu\alpha})(f(Y)-f(\mathsf{x}^\star))+\frac{1}{4}\|\dot{Y}\|^2+\frac{1}{4}\|\dot{Y}+ 2\sqrt{\mu}(Y-\mathsf{x}^\star)+\gamma(1+\sqrt{\mu\alpha})\sqrt{\alpha}\,\nabla f(Y)\|^2.
\end{align}
Here note that by definition of $\mathsf{x}^\star$, $f(\mathsf{x}^\star)\leq f(Y)$, with equality holding only at $Y=\mathsf{x}^\star$. 
Thus, $V(t)$ is positive everywhere, and zero only at $Y=\mathsf{x}^\star$ and $\dot{Y}=0$. The derivative  of Lyapunov function~\eqref{eq::v} along the trajectories $t\mapsto Y$ of~\eqref{eq::ODE_Y} (or equivalently~\eqref{eq::TM_ODE}) is 
\begin{align*}
\begin{split}
    &\dot{V}(t)=
     -\sqrt{\mu}\|\dot{Y}\|^2 - \frac{1}{2} \gamma(1+\sqrt{\mu\alpha})\sqrt{\alpha} \dot{Y}^\top \nabla^2f(Y)\dot{Y} 
    \\&\qquad\qquad\qquad\qquad- \sqrt{\mu}(1+\sqrt{\mu\alpha})\nabla f(Y)^\top(Y-\mathsf{x}^\star) 
    \\&\qquad\qquad\qquad\qquad- \frac{1}{2}\gamma(1+\sqrt{\mu\alpha})^2\sqrt{\alpha}\nabla f(Y)^\top\nabla f(Y).
\end{split}
\end{align*}
To show that $\dot{V}<0$ everywhere except at $Y=\mathsf{x}^\star$ and $\dot{Y}=0$, we consider the following relations. First, we note that it follows from~\eqref{eq::MCon-1} and~\eqref{eq::MCon-3} that
\begin{align*}
\begin{split}
  \left(1+\sqrt{\mu\alpha}\right)\nabla f(Y)^\top(Y-\mathsf{x}^\star) 
  &= \frac{(1+\sqrt{\mu\alpha})}{2}\nabla f(Y)^\top(Y-\mathsf{x}^\star) +\frac{1}{2}\nabla f(Y)^\top(Y-\mathsf{x}^\star)\\
  & \geq\frac{(1+\sqrt{\mu\alpha})}{2}\Big(\big( f(Y)-f(\mathsf{x}^\star)\!+\!\frac{M}{2}\| Y-\mathsf{x}^\star \|^2\Bigr)+ \\
  &\frac{M}{2} \| Y-\mathsf{x}^\star \|^2\geq \frac{1+\sqrt{\mu\alpha}}{2}\left( f(Y)-f(\mathsf{x}^\star) \right) + \frac{3M}{4} \| Y-\mathsf{x}^\star \|^2.
\end{split}
\end{align*}
Thus, we have 
\begin{align*}
    &\dot{V}(t) \leq  - \frac{\sqrt{\mu}}{2} \left(1+\sqrt{\mu\alpha}\right) \left( f(Y)-f(\mathsf{x}^\star) \right) -\sqrt{\mu}\|\dot{Y}\|^2
    - \frac{3M}{4} \sqrt{\mu} \| Y-\mathsf{x}^\star \|^2
   - \frac{1}{2}\gamma\sqrt{\alpha} \left(1+\sqrt{\mu\alpha}\right)^2 \left\| \nabla f(Y) \right\|^2
 \\&\qquad\leq  -\sqrt{\mu} \Bigl( \frac{1}{2} \left(1+\sqrt{\mu\alpha}\right) \left( f(Y)-f(\mathsf{x}^\star) \right) + \|\dot{Y}\|^2
   + \frac{3M}{4}  \| Y-\mathsf{x}^\star \|^2 + \frac{1}{2}\frac{\gamma\sqrt{\alpha}}{\sqrt{\mu}} \left(1+\sqrt{\mu\alpha}\right)^2 \left\| \nabla f(Y) \right\|^2 \Bigr).
\end{align*}


Next using the Young's inequality~\cite{UHY:12} for a $\phi>0$ we write
\begin{align*}
  &\frac{1}{4}\left\|\dot{Y}+2\sqrt{\mu}(Y-\mathsf{x}^\star)+\gamma(1+\sqrt{\mu\alpha})\sqrt{\alpha}\nabla f(Y) \right\|^2 \\
   &\qquad \qquad \leq \frac{1}{4}(1+\frac{1}{\phi})\|\dot{Y}+\gamma(1+\sqrt{\mu\alpha})\sqrt{\alpha}\,\nabla f(Y)\|^2 + \frac{(1+\phi)}{4}\|2\sqrt{\mu}(Y-\mathsf{x}^\star)\|^2\\
   &\qquad \qquad \leq\frac{1}{2}(1+\frac{1}{\phi})\|\dot{Y}\|^2+\frac{\gamma^2(1+\sqrt{\mu\alpha})^2\alpha}{2}(1+\frac{1}{\phi})\|\,\nabla f(Y)\|^2 +\mu(1+\phi)\|Y-\mathsf{x}^\star\|^2.
\end{align*}
Thus, we have
\begin{align*}
\begin{split}
  &\!\!\!\!\!V(t)\leq  (1+\sqrt{\mu\alpha})(f(Y)\!-\!f(\mathsf{x}^\star))+ \frac{1}{4}(3+\frac{2}{\phi})\|\dot{Y}\|^2+\frac{\gamma^2(1\!+\!\sqrt{\mu\alpha})^2\alpha}{2}(1\!+\!\frac{1}{\phi})\|\,\nabla f(Y)\|^2\!+\!\mu(1\!+\!\phi)\|Y\!-\!\mathsf{x}^\star\|^2\!.
\end{split}
\end{align*}



Now using $p(\phi)$ given in the theorem statement, we can write 
\begin{align*}
\begin{split}
    &p(\phi)\,V(t) \leq \sqrt{\mu}\left ( \frac{1}{2}(1+\sqrt{\mu\alpha})(f(Y)-f(\mathsf{x}^\star)) + \|\dot{Y}\|^2  + \frac{3M}{4} \|Y-\mathsf{x}^\star\|^2 +  \frac{\gamma\sqrt{\alpha}}{2\sqrt{\mu}}(1+\sqrt{\mu\alpha})^2 \|\nabla f(Y)\|^2 \right ).
\end{split}
\end{align*}
Therefore, $\dot{V}(t) \leq -p(\phi)\,V(t)$, for any $\phi\in\real_{>0}$. Then, we can conclude that $t\mapsto Y$  and $t\mapsto \dot{Y}$ converge asymptotically to, respectively $\mathsf{x}^\star$ and $0$. Next, we show that this convergence is indeed exponentially fast. To this end, using the Comparison Lemma~\cite[Lemma 3.4]{HKK:02} we obtain
\begin{align}\label{eq::dot_V}
    V(t) \leq \textup{e}^{-p^\star t}\,V(0).
\end{align}
Consequently, since $f(Y(t))-f(\mathsf{x}^\star)\leq \frac{1}{1+\sqrt{\mu\alpha}}V$, starting from any initial condition, we obtain $f(Y(t))-f(\mathsf{x}^\star)\leq \frac{1}{1+\sqrt{\mu\alpha}}\textup{e}^{-p^\star t}\,V(0)$, showing that $f(Y(t))-f(\mathsf{x}^\star)$ vanishes exponentially with a rate no worse than $p^\star$. Next, using~\eqref{eq::Lip-3} we note that (recall $\nabla f(\mathsf{x}^\star)=0$)
$$\frac{1}{2L}\|\nabla f(Y(t))\|^2\!\leq \! f(Y(t))-f(\mathsf{x}^\star)\!\leq\! \frac{1}{1\!+\!\sqrt{\mu\alpha}}\textup{e}^{-p^\star t}V(0), $$
which indicates that $\nabla f(Y(t))$ converges exponentially to zero. On the other hand, using~\eqref{eq::MCon-3} we can write 
\begin{align*}
  &  \frac{M^2}{2L}\|Y-\mathsf{x}^\star\|^2\leq\frac{1}{2L}\|\nabla f(Y(t))\|^2\leq f(Y(t))-f(\mathsf{x}^\star),
    \end{align*}
to conclude that $Y$ converges exponentially to $\mathsf{x}^\star$. To prove exponential convergence of $t\mapsto\varepsilon $ and $t\mapsto X$ to $\mathsf{x}^\star$ we proceed as follows. 
We let $\eta = \dot{\varepsilon}$. Now from~\eqref{eq::ODE_1}, we have 
\begin{align*} 
\dot{\eta} = - 2\sqrt{\mu} \,\eta - (1+\sqrt{\mu\alpha}) \nabla f(Y),
\end{align*}
which is an internally exponentially stable system with input $\nabla f(Y)$ driven by~\eqref{eq::ODE_Y}.  Since $\nabla f(Y)$ converges exponentially 
to zero, due to the input-to-state stability results~\cite{HKK:02}, we can conclude that $\eta$ (equivalently $\dot{\varepsilon}$) converges exponentially fast 
to $0$. As a result, it follows from~\eqref{eq::ODE_2},~\eqref{eq::ODE_3} and exponential 
convergence of $Y$ to $\mathsf{x}^\star$ that $t\mapsto X$ and $t\mapsto \varepsilon$ also converge to $\mathsf{x}^\star$, exponentially fast. 

Next, we note that under the initial condition~\eqref{eq::initial_TM}, by substitution we obtain
\begin{align*}
    Y(0) &= \varepsilon(0) + \sqrt{\alpha}\gamma\,\dot{\varepsilon}(0)=Y_0.
\end{align*}
Moreover, 
\begin{align*}
    \dot{Y}(0) &=   \dot{\varepsilon}(0) + \sqrt{\alpha}\gamma\ddot{\varepsilon}(0) \\
    &= \dot{\varepsilon}(0) - \sqrt{\alpha}\gamma \left( 2\sqrt{\mu}\dot{\varepsilon}(0)+(1+\sqrt{\mu\alpha})\nabla f(Y_0)\right)=0.
\end{align*}
Substituting the initial condition $Y(0)=Y_0$ and $\dot{Y}(0)=0$ in~\eqref{eq::v}, we get 
\begin{align}\label{eq::V_init}
V(0)=&\,(1+\sqrt{\mu\alpha})(f(Y_0)-f(\mathsf{x}^\star))+\frac{1}{4}\|2\sqrt{\mu}(Y_0-\mathsf{x}^\star)+\gamma(1+\sqrt{\mu\alpha})\sqrt{\alpha}\nabla f(Y_0)\|^2.
\end{align}
After substituting for $V(t)$ from~\eqref{eq::v} and $V(0)$ from~\eqref{eq::V_init}, it follows from~\eqref{eq::dot_V} that 
\begin{align*}
\begin{split}
&f(Y)-f(\mathsf{x}^\star)\leq \textup{e}^{-p^{\star}t}\Bigr(f(Y_0)-f(\mathsf{x}^\star)+\frac{1}{4(1\!+\!\sqrt{\mu\alpha})}\|2\sqrt{\mu}(Y_0-\mathsf{x}^\star)\!+\!\gamma(1\!+\!\sqrt{\mu\alpha})\sqrt{\alpha}\nabla f(Y_0)\|^2\Bigr).
\end{split}
\end{align*}
Note that by invoking $\|\nabla f(Y_0)\|\leq L\|Y_0-\mathsf{x}^\star\|$ and  $f(Y)-f(\mathsf{x}^\star)\leq \frac{L}{2}\|Y_0-\mathsf{x}^\star\|^2$, which hold for function $f$ by its definition, we obtain
\begin{align*}
\begin{split}
\|2\sqrt{\mu}(Y_0-\mathsf{x}^\star)+\gamma(1+\sqrt{\mu\alpha})\sqrt{\alpha}\nabla f(Y_0)\|^2 &\leq  2\left(4\mu\|Y_0-\mathsf{x}^\star\|^2+\gamma^2(1+\sqrt{\mu\alpha})^2\alpha\|\nabla f(Y_0)\|^2\right)\\
&\leq (8\mu+2\gamma^2(1+\sqrt{\mu\alpha})^2\alpha L^2)\|Y_0-\mathsf{x}^\star\|^2.
\end{split}
\end{align*}
Thus we have
\begin{align*}
& f(Y)-f(\mathsf{x}^\star)\leq\left(\frac{L}{2}+\frac{2\mu}{(1+\sqrt{\mu\alpha})}
 +\frac{\gamma^2(1+\sqrt{\mu\alpha})\alpha L^2}{2}\right)\|Y_0-\mathsf{x}^\star\|^2\textup{e}^{-p^{\star}t}.
\end{align*}
Note that using the parameters of the TM method in~\eqref{eq::TM_param} we can write
\begin{align*}
&\frac{L}{2}+\frac{2\mu}{(1+\sqrt{\mu\alpha})}+\frac{\gamma^2(1+\sqrt{\mu\alpha})\alpha L^2}{2}
\\&\qquad\qquad\quad=\frac{(1.5\rho^4+3\rho^3-3.5\rho^2-4\rho+6)}{\alpha(-\rho^3+3\rho^2-4\rho+4)}\xrightarrow[]{\rho\rightarrow1} \frac{1.5}{\alpha}.
\end{align*}
Therefore, $f(Y)-f(\mathsf{x}^\star)\leq \frac{1.5\|Y_0-\mathsf{x}^\star\|}{\alpha}^2\textup{e}^{-p_{_{\textup{TM}}}^\star t}$, which completes the proof.
\end{proof} 

Theorem~\ref{thm::main} shows that~\eqref{eq::TM_ODE} has robustness to deviations from the TM parameters. But, an interesting observation  about our rate of convergence analysis is that our simulation study of the rate $p$ in Theorem~\ref{thm::main} indicates that the best rate is obtained when we use $\alpha,\beta,\gamma$ of the TM method given in~\eqref{eq::TM_param}, see Fig.~\ref{fig::p_vs_alpha} for some example scenarios.
\begin{figure}[h!]
  \centering  
  \includegraphics[trim=10pt 33pt 0pt 10pt, clip,scale=0.75]{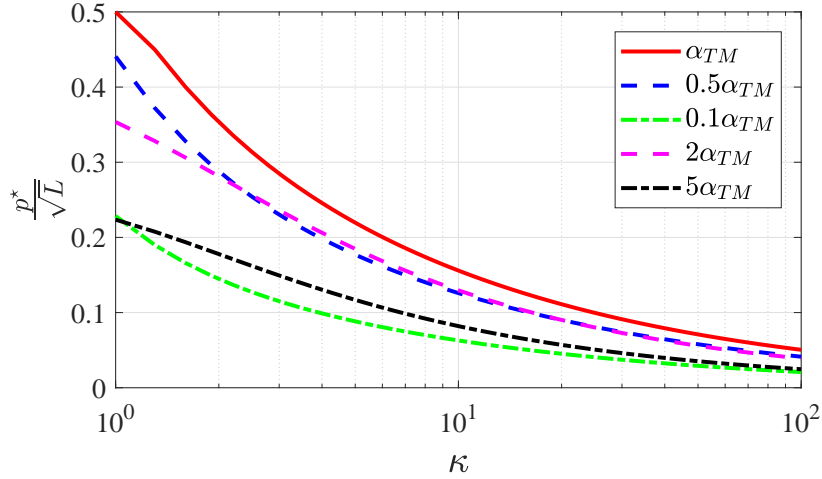}
  \caption{Variation of $p^\star/\sqrt{L}$, where $p^\star$ is given in~\eqref{eq::p_TM}, with $\kappa$ when we use $\beta$ and $\gamma$ of the TM method but implement different values for $\alpha$. $\alpha_{TM}$ corresponds to the $\alpha$ of the TM method.}
  \label{fig::p_vs_alpha}
\end{figure}

Next, we note that the rate of convergence established for the ODE representation of the  NAG method~\eqref{eq::ODE_Y_NAG} in~\cite{BS:18} is $\frac{\sqrt{M}}{4}$. Before, comparing this rate to the rate of the TM method that we established in Theorem~\ref{thm::main}, we show that a tighter bound can be indeed obtained for the NAG method. For brevity, we skip the details and comment only on the crucial steps required to establish this tighter bound. We start by using the Lyapunov candidate function $V$ defined as
\begin{align*}
  V&= (1+\sqrt{Ms})(f(Y)-f(\mathsf{x}^\star))+\frac{1}{4} \|\dot{Y}\|^2\\
  &\quad\qquad\qquad\qquad+ \frac{1}{4}\|\dot{Y}+\sqrt{s}\,\nabla f(Y)+\sqrt{M}(Y-\mathsf{x}^\star)\|^2\\
  &\leq  (1+\sqrt{Ms})(f(Y)-f(\mathsf{x}^\star))+ \frac{1}{4}(3+\frac{2}{\phi})\|\dot{Y}\|^2\\
  &\quad\qquad+\frac{s}{2}(1+\frac{1}{\phi})\|\,\nabla f(Y)\|^2+M(1+\phi)\|Y-\mathsf{x}^\star\|^2\!,
\end{align*}
where $\phi\in\real_{>0}$. The upper-bound on $V$ is established using the similar manipulations we used in the proof of Theorem~\ref{thm::main}. Now following similar steps given in the proof of Theorem~\ref{thm::main} we can show that derivative of $V$ along trajectories of ~\eqref{eq::ODE_Y_NAG} satisfies
\begin{align*}
  \dot{V}&\leq  - \frac{\sqrt{M}}{2} (1+\sqrt{Ms}) \big( f(Y)-f(\mathsf{x}^\star) \big) -\sqrt{M}\|\dot{Y}\|^2
  \\
  &\qquad\qquad\quad-\frac{3M\sqrt{M}}{4}  \| Y-\mathsf{x}^\star \|^2
   - \frac{s\sqrt{M}}{2} \left\| \nabla f(Y) \right\|^2.  
\end{align*}
Thus, for the ODE NAG method, the convergence rate $p_{_{\textup{NAG}}}^\star$ is 
\begin{align}
&p_{_{\textup{NAG}}}^\star=\max_{\phi\in\real_{>0}}p_{_{\textup{NAG}}}(\phi),\\
 & p_{_{\textup{NAG}}}(\phi) \!=\! \min{\small \left\{\!\frac{\sqrt{L}}{2\sqrt{\kappa}},\frac{3\sqrt{L}}{4\sqrt{\kappa}(1\!+\!\phi)},\!\frac{\sqrt{L}}{\sqrt{\kappa}(1\!+\!\frac{1}{\phi})},\!\frac{4\sqrt{L}}{\sqrt{\kappa}(3\!+\!\frac{2}{\phi})}\right\}}.\nonumber
\end{align}
Given $\frac{L}{\kappa}=M$, we can write  $$p_{_{\textup{NAG}}}(\phi) = \min \left\{\frac{1}{2},\frac{3}{4(1\!+\!\phi)},\frac{1}{(1\!+\!\frac{1}{\phi})},\frac{4}{(3\!+\!\frac{2}{\phi})}\right\}\sqrt{M}.$$ Figure~\ref{fig::P_NAG_phi} shows how each of the four elements varies with $\phi$ and the optimal $\phi$ for which the minimum among the four elements is at its maximum. As can be seen and also shown analytically $p_{_{\textup{NAG}}}^\star=\frac{3}{7}\sqrt{M}$ is attained at $\phi^\star=\frac{3}{4}=0.75$. Thus, $p_{_{\textup{NAG}}}^\star$ is a tighter bound than $\frac{\sqrt{M}}{4}$ that is established in~\cite{BS:18} as the rate of convergence for the ODE NAG method. On the other hand, Fig.~\ref{fig::Compare_P_TM_NAG} compares $\frac{p_{_\text{TM}}}{\sqrt{L}}$ with $\frac{p_{_\text{NAG}}}{\sqrt{L}}$ at different values of $\kappa$. As we can see the TM method attains a better convergence rate than the NAG method. In comparing the rate of convergences of the TM and  NAG methods, it is worth to remember~\eqref{eq::mu_kappa_L} and~\eqref{eq::mu-L}. It is also interesting to note that similar to the gradient descent method, the rate of convergence of the TM and the NAG methods decreases as $\kappa$ increases. Finally note that $\phi^\star$ corresponding to $p^\star_{_\text{TM}}$ can be obtained as 
$$\phi^{\star} = \frac{9L-16\mu \kappa+\sqrt{256(\mu\kappa)^2+96\mu kL+81L^2}}{32\mu k}.$$

\begin{figure}[h!]
  \begin{center} 
  \includegraphics[trim=10pt 33pt 10pt 10 ,clip,scale=0.75]{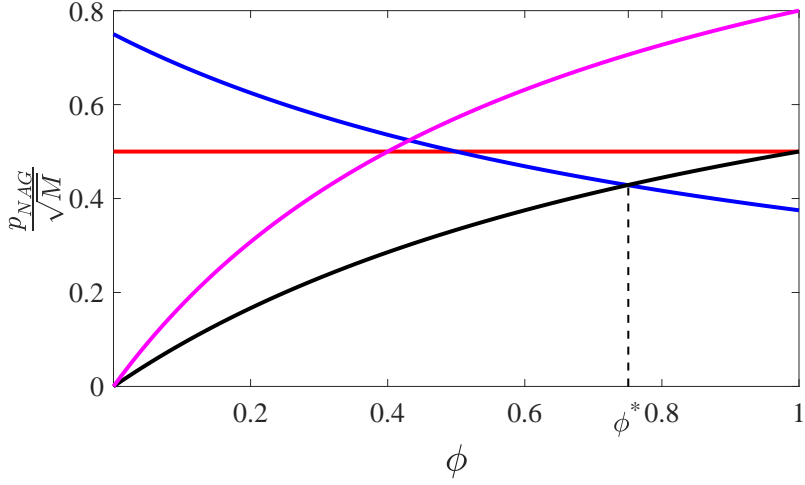}
  \end{center}
  \caption{Variation of the elements of {\small$\left\{\frac{1}{2},\frac{3}{4(1+\phi)},\frac{1}{(1+\frac{1}{\phi})},\frac{4}{(3+\frac{2}{\phi})}\right\}$} with $\phi$. $\frac{p^\star_{_\text{NAG}}}{\sqrt{M}}=\frac{3}{7}=0.4286$ is attained at $\phi^\star=0.75$. }
  \label{fig::P_NAG_phi}
\end{figure}

\begin{figure}[h!]
  \centering  
  \includegraphics[trim=10pt 33pt 0pt 10pt, clip,scale=0.75]{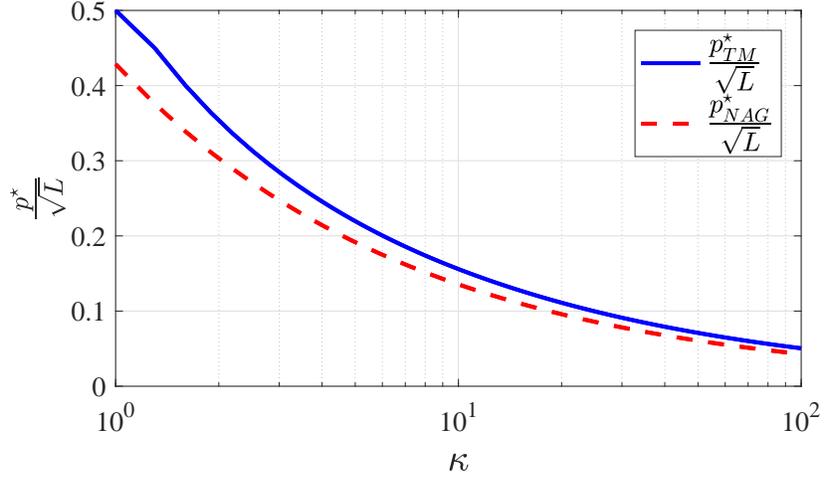}
  \caption{A comparison between the rate of convergence of the ODE representations of the TM and  the NAG methods at different values of $\kappa$.}
  \label{fig::Compare_P_TM_NAG}
\end{figure}

\subsection{Analysis via IQC}
We close this section by noting that the rate of convergence of the continuous-time TM can be also obtained using an IQC method. To this end, note that algorithm~\eqref{eq::TM_ODE} can be cast as an LTI system
\begin{subequations}\label{LTI}
\begin{align}
      \dot{\xi}(t)=A\xi(t)+Bq(t)\\
    Y(t)=C\xi(t)+Dq(t)
\end{align}
\end{subequations}
with state $\xi(t)=[\dot{\varepsilon}(t)\quad \varepsilon(t)]^{\top}\in \mathbb{R}^{2n}$, input $q(t)=\nabla f(Y(t))\in \mathbb{R}^{2n}$, and output $Y(t)\in \mathbb{R}^{2n}$, where
$A=\begin{bmatrix} 
-2\sqrt{\mu}&0\\
1 & 0\\
\end{bmatrix}\otimes I_n,\quad
B=\begin{bmatrix}
-1-\sqrt{\mu\alpha}&0\\
0&0\\
\end{bmatrix}\otimes I_n$, 
$
C=\begin{bmatrix} 
\sqrt{\alpha}\gamma& 1\\
\sqrt{\alpha}\delta&1\\
\end{bmatrix}^{\top}\otimes I_n, \quad
D=0_{2n\times 2n}
. $
When $f\in\mathcal{S}_{M,L}$, ~\cite{Lessard:18} shows that the nonlinear map $q(t)=\nabla f(Y)$ satisfies the so-called point-wise IQC condition cast as
 \begin{align}\label{eq::IQC_TM_con}
     \begin{bmatrix}
     Y-Y^\star\\\nabla f(Y)-\nabla f(Y^\star)
     \end{bmatrix}^\top {Q_f} \begin{bmatrix}
     Y-Y^\star\\\nabla f(Y)-\nabla f(Y^\star)
     \end{bmatrix}\succeq 0,
 \end{align}
where $Q_f=\begin{pmatrix}
-2ML&L+M\\
L+M&-2
\end{pmatrix}\otimes I_n$, and $Y^\star=\mathsf{x}^\star$.

\begin{rem}[An estimate on the rate of convergence of $Y$ in~\eqref{eq::TM_ODE} using an IQC based solution]\label{rem::IQC_conv}\cite{ZEN:18}
\rm{Given the point-wise IQC condition for the LTI representation of the continuous-time TM, using standard IQC stability results, the exponential convergence rate of $\|Y(t)-\mathsf{x}^\star\|$ to zero in the continuous-time TM algorithm~\eqref{eq::TM_ODE} is $p_{_\text{IQC}}$ if
\begin{align}\label{eq::LMI}
\begin{bmatrix} 
A^{\top}P\!+\!PA\!+\!p_{_\text{IQC}} P & PB \\
B^{\top}P & 0
\end{bmatrix}\!+\!\sigma
\begin{bmatrix} 
C^{\top} & 0 \\
D^{\top} & I
\end{bmatrix}Q_{f}
\begin{bmatrix} 
C & D \\
0 & I
\end{bmatrix}\!\preceq\!0
\end{align}
is feasible for some $\sigma\in\real_{\geq0}$, $p_{_\text{IQC}}\in\real_{>0}$, $P\succ0$, $P \in \reals^{n\times n}$. A tighter estimate $p^\star_{_\text{IQC}}$ on the rate of convergence can be obtained by maximizing $p_{_\text{IQC}}$ subject to~\eqref{eq::LMI}.}
 \boxend
 \end{rem}
 \begin{figure}[t]
  \centering  
  \includegraphics[trim=10pt 33pt 0pt 10pt, clip,scale=0.75]{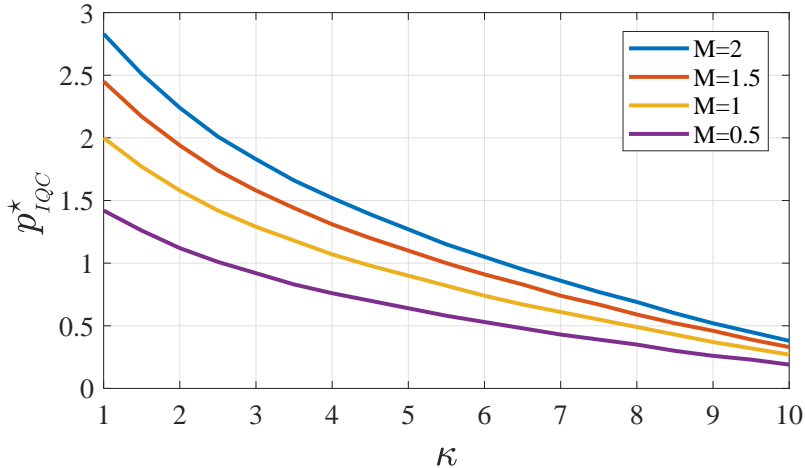}
  \caption{Convergence rate of the TM method given by the IQC method of Remark~\ref{rem::IQC_conv}.}
  \label{Fig::iqc} 
\end{figure}
Figure~\ref{Fig::iqc} shows the convergence rate $p^\star_{_\text{IQC}}$ that we get from using the IQC method of Remark~\ref{rem::IQC_conv} for various values of $M$ and $\kappa$. As we can see, the IQC approach also shows that similar to the gradient descent method, the rate of convergence of the TM method also decreases as $\kappa$ increases. We should mention here though that the IQC approach offers a sufficient condition for stability and convergence analysis, which is not guaranteed to yield a solution for every value of $M$ and $L$. 

\section{Simulation results}\label{sec:num}
\begin{figure}[ht]
  \begin{centering}
      \subfigure[When the parameters of the algorithms are set to their respective exact values]{
      \psfrag{convergence error in log scale}{\tiny{convergence error in log scale}}
      \psfrag{time}{\tiny{time}}
      \includegraphics[trim=10pt 33pt 5pt 10pt, clip,scale=0.85]{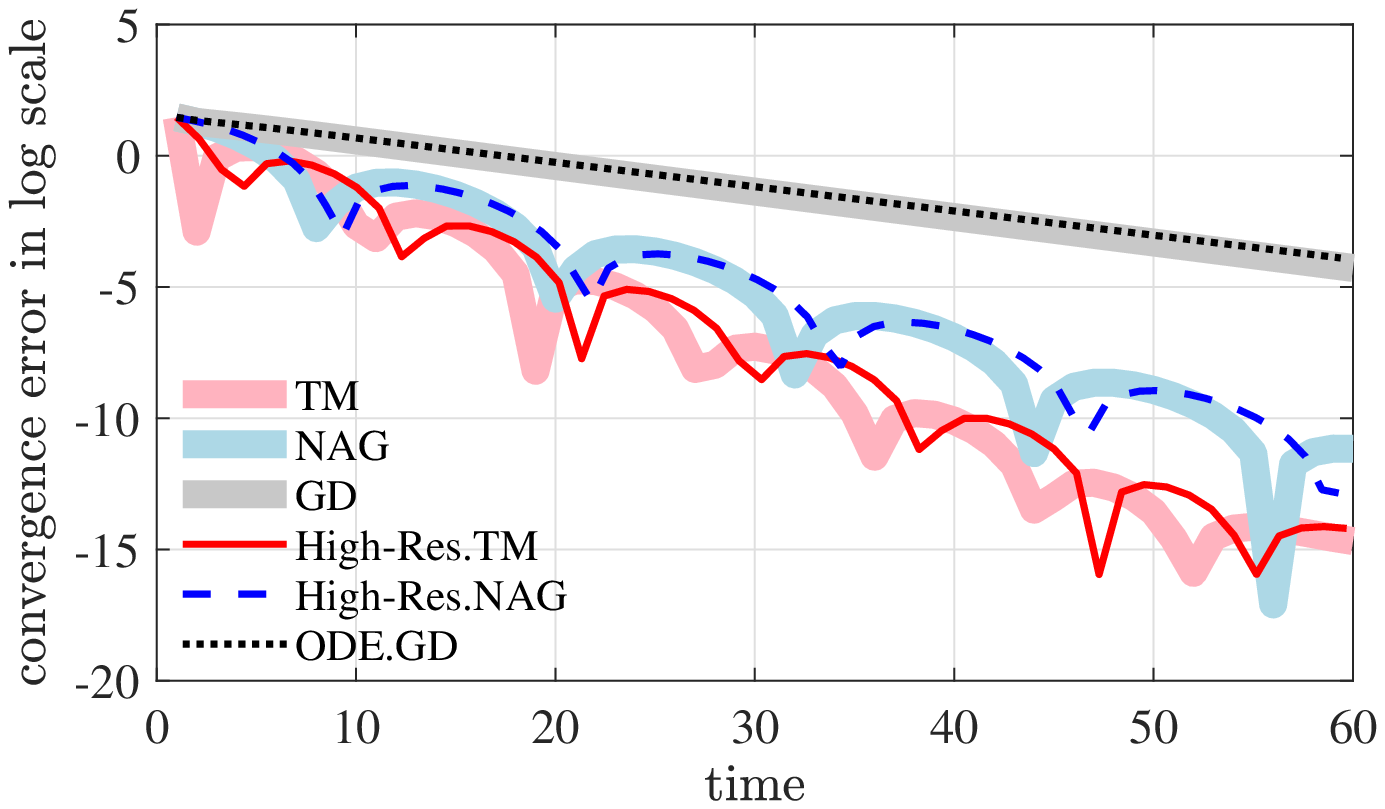}\label{Fig::exm1}}\\
      \subfigure[When the parameters of the algorithms are set to their respective exact values except for stepsizes which are scaled down by a factor of $0.3$]{
      \includegraphics[trim=10pt 33pt 5pt 10pt, clip,scale=0.85]{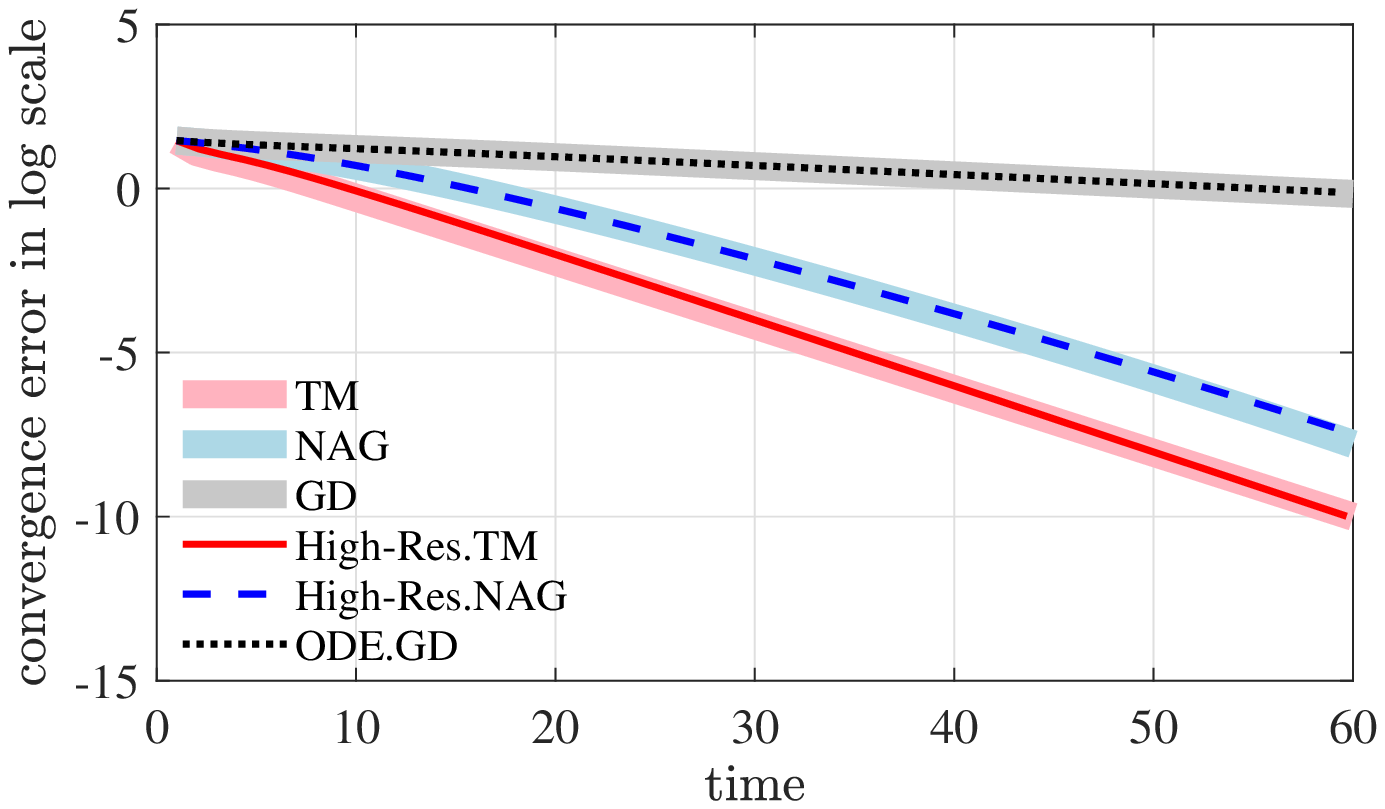}\label{Fig::exm2}}
      \caption{Convergence error for the TM, NAG and gradient descent (GD) algorithms.}
      \label{Result3}
  \end{centering}
\end{figure}
Let the cost function in~\eqref{Eqn:Opt} be given by $ f(x)=\frac{x^2}{2\,\text{log}(2+x^2)}-x$. For this cost, we have $M=0.038$ and $L=1.443$. Thus, $\kappa=37.713$.
Figure~\ref{Fig::exm1} shows the convergence error for the TM, NAG, gradient descent with stepsize $1/L$ (GD), high-resolution ODE representations  of TM~\eqref{eq::TM_ODE} and  NAG~\eqref{eq::ODE_Y_NAG} methods, and continuous-time gradient descent (ODE GD) algorithms. Figure~\ref{Fig::exm2} shows the same plot when a smaller stepsize is used for all the algorithms. As we can see in these plots, the high-resolution ODE representation of the TM algorithm closely captures the characteristics of the discrete-time TM. Moreover, we can see from the plots that for both cases, the TM algorithm converges faster than the gradient descent and the NAG methods. We can also see that using a smaller stepsize removes the oscillatory behavior that we see in the trajectories of the TM and NAG methods however as expected and predicted by our analysis using a smaller stepsize results in a slower convergence. 

\section{Conclusion}
In this paper, we have presented a second-order ODE for modeling the triple momentum method, which is considered as the fastest first-order optimization method for strongly convex functions. The proposed high-resolution ODE model has shown to accurately captures the higher-order characteristics of its discrete-time counterpart. We presented a Lyapunov analysis to prove the exponential convergence of the developed continuous-time model of the triple momentum algorithm. We compare the rate of this ODE model of the triple momentum with that of the Nesterov method and showed that the Lyapunov analysis also confirms that the triple momentum method has a faster convergence than the Nesterov method.  We also discuss how an IQC approach also can be used to obtain an estimate on the rate of convergence of the ODE representation of the triple momentum method. 
We validate our theoretical results through several numerical simulations. Since control theoretic tools in continuous-domain generally provide a convenient framework for design and analysis of algorithms, our future work includes first devising a distributed version of the continuous-time triple momentum method that can be used for distributed optimizations. Then, our objective is to discretize this algorithm to obtain an iterative solution that can be implemented over networks with wireless communication.

\bibliographystyle{IEEEtran}%
\bibliography{alias,Reference,Biblio}

\end{document}